\pdfoutput=1

\documentclass[reqno,a4paper,final]{amsart}

\usepackage[utf8]{inputenc}
\usepackage[T1]{fontenc}
\usepackage[english]{babel}
\usepackage{ifthen}
\usepackage{enumitem}
\usepackage{dsfont}
\usepackage{mathtools}
\usepackage{amssymb}
\usepackage[babel]{csquotes}
\usepackage[backend=biber,style=alphabetic]{biblatex}
\usepackage[protrusion=true,expansion=true,babel=true,final]{microtype}
\usepackage[unicode,bookmarksopen]{hyperref}

\numberwithin{equation}{section} 

\newtheorem{lemma}{Lemma}[section]
\newtheorem{corollary}[lemma]{Corollary}
\newtheorem{proposition}[lemma]{Proposition}
\newtheorem{theorem}[lemma]{Theorem}

\theoremstyle{definition}
\newtheorem{definition}[lemma]{Definition}
\newtheorem{example}[lemma]{Example}
\newtheorem{remark}[lemma]{Remark}

\newlist{thm_enum}{enumerate}{1}
\setlist[thm_enum]{label=\normalfont(\alph*)}
\newlist{def_enum}{enumerate}{1}
\setlist[def_enum]{label=\normalfont(\roman*)}
\newlist{equiv_enum}{enumerate}{1}
\setlist[equiv_enum]{label=\normalfont(\roman*)}

\newcommand{\IN}{\mathbb{N}}
\newcommand{\IR}{\mathbb{R}}
\newcommand{\IC}{\mathbb{C}}
\newcommand{\IE}{\mathbb{E}}

\newcommand{\IP}{\mathbb{P}}

\newcommand{\abs}[1]{\left\lvert#1\right\rvert}
\newcommand{\normalabs}[1]{\lvert#1\rvert}

\newcommand{\norm}[1]{\left\lVert#1\right\rVert}
\newcommand{\normalnorm}[1]{\lVert#1\rVert}
\newcommand{\biggnorm}[1]{\biggl\lVert#1\biggr\rVert}

\newcommand{\R}[2][\empty]{
	\ifthenelse{\equal{#1}{\empty}}
		{\mathcal{R}\left\{#2\right\}}
		{\mathcal{R}_{#1}\left\{#2\right\}}
}

\makeatletter
\newcommand{\LeftEqNo}{\let\veqno\@@leqno}
\makeatother

\renewcommand{\d}{\mathop{}\!d}
\renewcommand{\Re}{\operatorname{Re}}

\renewcommand{\epsilon}{\varepsilon}

\let\temp\phi
\let\phi\varphi
\let\varphi\temp

\DeclareMathOperator{\E}{\IE}

\DeclareMathOperator{\Id}{Id}

\DeclareMathOperator{\Rad}{Rad}

\DeclareMathOperator{\divv}{div}

\DeclareMathOperator{\VMO}{VMO}
\DeclareMathOperator{\BMO}{BMO}
\DeclareMathOperator{\BUC}{BUC}

\allowdisplaybreaks[4]

\DeclareSourcemap{
  \maps[datatype=bibtex, overwrite]{
    \map{
      \step[fieldset=abstract, null]
      \step[fieldsource=entrykey,match=Mer99,final] %
      \step[fieldset=shorthand,fieldvalue=LeM99] %
    }
  }
}

\renewbibmacro{publisher+location+date}{%
	\printlist{publisher}
	\setunit*{\addcomma\space}
	\printlist{location}
  	\setunit*{\addcomma\space}
  	\usebibmacro{date}
\newunit} 

\newbibmacro{string+doi+url}[1]{%
	\iffieldundef{doi}{%
			\iffieldundef{url}{#1}{\href{\thefield{url}}{#1}}%
		}%
		{\href{http://dx.doi.org/\thefield{doi}}{#1}}
}%

\renewbibmacro{in:}{}

\DeclareFieldFormat*{title}{\usebibmacro{string+doi+url}{\mkbibemph{#1}}}
\DeclareFieldFormat*{booktitle}{#1}
\DeclareFieldFormat[article]{volume}{\textbf{#1}\addcomma}
\DeclareFieldFormat[article]{number}{\addspace no.~#1}
\DeclareFieldFormat[article]{pages}{#1}
\DeclareFieldFormat{journaltitle}{#1} 
\DeclareFieldFormat{url}{}
\DeclareFieldFormat{eprint}{Preprint on arXiv: \href{http://arxiv.org/abs/#1}{#1}} 

\renewcommand*{\bibfont}{\footnotesize}

\begin{document}

\title[Maximal Regularity and Fractional Sobolev Regularity]{Non-Autonomous Maximal $L^p$-Regularity under Fractional Sobolev Regularity in Time}

\begin{abstract}
	We prove non-autonomous maximal $L^p$-regularity results on UMD spaces replacing the common Hölder assumption by a weaker fractional Sobolev regularity in time. This generalizes recent Hilbert space results by Dier and Zacher. In particular, on $L^q(\Omega)$ we obtain maximal $L^p$-regularity for $p \ge 2$ and elliptic operators in divergence form with uniform $\VMO$-modulus in space and $W^{\alpha,p}$-regularity for $\alpha > \frac{1}{2}$ in time.
\end{abstract}

\author{Stephan Fackler}
\address{Institute of Applied Analysis, Ulm University, Helmholtzstr.\ 18, 89069 Ulm}
\email{stephan.fackler@uni-ulm.de}
\thanks{This work was supported by the DFG grant AR 134/4-1 ``Regularität evolutionärer Probleme mittels Harmonischer Analyse und Operatortheorie''.}

\maketitle

\section{Introduction}

	In this work we improve some known results on maximal $L^p$-regularity of non-autonomous abstract Cauchy problems with time-dependent domains of the form
	\begin{equation*}
		\LeftEqNo
		\label{eq:nacp}\tag{NACP}
		\left\{
		\begin{aligned}
			\dot{u}(t) + A(t)u(t) & = f(t) \\
			u(0) & = u_0.
		\end{aligned}
		\right.
	\end{equation*}
	In particular, we obtain new stronger results if the operators $A(t)$ are elliptic operators in divergence form. For a family $(A(t))_{t \in [0,T]}$ of linear operators on some Banach space $X$ the problem~\eqref{eq:nacp} has \emph{maximal $L^p$-regularity} if for all $f \in L^p([0,T];X)$ and all initial values $u_0$ in the real interpolation space $(D(A(0)),X)_{1/p,p}$  there exists a unique solution $u \in L^p([0,T];X)$ satisfying $u(t) \in D(A(t))$ for almost all $t \in [0,T]$ as well as $\dot{u}, A(\cdot)u(\cdot) \in L^p([0,T];X)$ and if there exists $C > 0$ such that the maximal regularity estimate
	\begin{equation*}
		\norm{u}_{W^{1,p}([0,T];X} + \norm{A(\cdot)u(\cdot)}_{L^p([0,T];X)} \le C (\norm{f}_{L^p([0,T];X)} + \norm{u_0}_{(D(A(0)),X)_{1/p,p}})
	\end{equation*}
	holds. Observe that $W^{1,p}([0,T];X) \hookrightarrow C([0,T];X)$ and therefore the initial condition makes sense. Maximal regularity results have profound applications to non-linear parabolic problems as we will exemplify in Section~\ref{sec:applications}.
	
	We now give a summary of the previously known results on maximal $L^p$-regularity. The autonomous case $A(t) = A$ is well understood. Here, maximal $L^p$-regularity holds for one $p \in (1, \infty)$ if and only if it holds for all $p \in (1, \infty)$. Further, maximal $L^p$-regularity for $u_0 = 0$ implies maximal $L^p$-regularity for all $u_0 \in D(A(0),X)_{1/p,p}$. On Hilbert spaces an operator $A$ has maximal regularity if and only if $-A$ generates an analytic semigroup. In non-Hilbert spaces, not every generator of an analytic semigroup has maximal regularity, see~\cite{KalLan00} or~\cite{Fac14}. Here, an additional $\mathcal{R}$-boundedness assumption is needed. We refer to~\cite{DHP03} and~\cite{KunWei04} for details.
	
	Let us come to the non-autonomous case. Here the best understood setting is that of non-autonomous forms on Hilbert spaces. For this let $V, H$ be two Hilbert spaces with a dense embedding $V \hookrightarrow H$. A mapping $a\colon [0,T] \times V \times V \to \IC$ is called a \emph{coercive, bounded sesquilinear form} if $a(t,\cdot,\cdot)$ is sesquilinear for all $t \in [0,T]$ and if there exist $\epsilon, M > 0$ such that for all $u,v \in V$
	\begin{align*}
		\Re a(t,u,u) & \ge \epsilon \norm{u}_V^2, \\
		\abs{a(t,u,v)} & \le M \norm{u}_V \norm{v}_V.
	\end{align*}
	This induces operators $\mathcal{A}(t)\colon V \to V'$. We denote the parts in $H$ by $A(t)$. It has been shown in~\cite{HaaOuh15} that the operators $(A(t))_{t \in [0,T]}$ satisfy maximal $L^p$-regularity for all $p \in (1,\infty)$ if $t \mapsto \mathcal{A}(t)$ is $\alpha$-Hölder continuous for $\alpha > \frac{1}{2}$. For maximal $L^2$-regularity this has been improved to the fractional Sobolev regularity $t \mapsto \mathcal{A}(t) \in \dot{W}^{\alpha,2}([0,T];\mathcal{B}(V,V'))$ for $\alpha > \frac{1}{2}$~\cite{DieZac16}. If one consider elliptic operators
	\begin{equation*}
		L(t) = -\divv(A(t)\nabla \cdot)
	\end{equation*}
	for coefficients $A(t) = (a_{ij}(t))$, this translates into the regularity of the mappings $t \mapsto a_{ij}(t,\cdot) \in L^{\infty}$, i.e.\ $t \mapsto a_{ij}(t,\cdot) \in \dot{W}^{\alpha,2}([0,T];L^{\infty})$. The less regularity one needs here, the more applicable the results are to non-linear questions in form of a priori estimates. In the special case of elliptic operators in divergence form some more refined results are available, see~\cite{AusEge16} and~\cite{Fac16d}. However, all results have in common that one needs some differentiability in time of order at least $\frac{1}{2}$. This is no coincidence. Recent counterexamples to Lions' problem by the author~\cite{Fac16c} show that maximal $L^p$-regularity can fail if $t \mapsto C^{1/2}(\mathcal{B}(V,V'))$. However, dealing with non-linear problems one needs some form of Sobolev embedding to carry out the usual iteration procedure. In higher dimensional cases maximal regularity on $X = L^2$ is too weak for the embeddings to hold. Therefore one is interested in maximal regularity on $X = L^q$ for $q$ big enough.
	
	Non-autonomous maximal $L^p$-regularity on Banach spaces is far more involved. The classical works for time-dependent domains are~\cite{HieMon00} and~\cite{HieMon00b}. Although the general method used there is applicable on Banach spaces, maximal $L^p$-regularity was first only obtained on Hilbert spaces in a non-form setting \cite{HieMon00} and in~\cite{HieMon00b} extrapolated to $X = L^q$ for elliptic operators  assuming $a_{ij} \in C^{1/2}([0,T];C^1(\overline{\Omega}))$. A true generalization of this approach to Banach spaces was obtained in~\cite{PorStr06} using the emerging concept of $\mathcal{R}$-boundedness. Already the results in~\cite{HieMon00b} indicate a fundamental new issue in the non-Hilbert space setting. Whereas on $L^2$ the coefficients only need to be measurable in space, on $L^q$ all known results require some regularity in space. Recently, the author lowered the needed regularity in space and showed maximal $L^p$-regularity on $L^q$ for elliptic operators in divergence form if the coefficients have a uniform $\VMO$-modulus~\cite{Fac15c}.
	
	The aim of this work is to generalize both the results in~\cite{DieZac16} and~\cite{Fac15c}. We show maximal $L^p$-regularity on Banach spaces assuming fractional Sobolev regularity as in~\cite{DieZac16}. The obtained results are even new in the Hilbert space case as~\cite{DieZac16} fully relies on Hilbert space methods and therefore only deasl with the case $p = 2$. In the case of elliptic operators in divergence form we require the coefficients apart from the $\VMO$ assumption to be in $\dot{W}^{\frac{1}{\alpha},p}([0,T]; L^{\infty})$ for some $\alpha > \frac{1}{2}$. This lowers the regularity needed in time for the treatment of non-linear problems and is the first improvement of the time regularity on general Banach spaces since the classical work~\cite{AcqTer87}. Since we establish maximal $L^p$-regularity for elliptic operators on $L^q(\Omega)$ for $q > 2$, we obtain existence results for \emph{strong} solutions of quasilinear parabolic equations in divergence form. Such results cannot be obtained with maximal regularity results on Hilbert spaces. We further show that our results are optimal in the sense that in general we can not relax the regularity to some $\alpha \le \frac{1}{2}$.
	
	Note that in contrast elliptic operators in non-divergence form have time independent domains and one can therefore obtain maximal $L^p$-regularity only assuming the time dependence to be measurable, see for example~\cite{GalVer14},~\cite{DonKim16} and the references therein for recent results. However, note that in correspondence with our results one still needs something like $\VMO$-regularity in space.
	
	This work is structured as follows. In the first sections we work towards an abstract maximal regularity result proven in Theorem~\ref{thm:mr}. As a consequence, we obtain in Theorem~\ref{thm:mr_elliptic} the stated result for elliptic operators. Section~\ref{sec:applications} uses this result to establish strong solutions of quasilinear elliptic equations. We discuss the optimality of our results in Section~\ref{sec:optimality}.
	
\section{The Fundamental Identity}
	
	Using ideas established by Acquistapace and Terreni in~\cite{AcqTer87} and previous works, we show in this section that maximal $L^p$-regularity solutions of~\eqref{eq:nacp} satisfy a certain integral equation. It turns out that this equation is better approachable with analytic tools. We recall some basic definitions first. For $\phi \in (0, \pi)$ we denote by $\Sigma_{\phi} \coloneqq \{ z \in \IC \setminus \{ 0 \}: \abs{\arg z} < \phi \}$ the sector of angle $\phi$.
	
	\begin{definition}\label{def:sectorial}
		A linear operator $A\colon D(A) \to X$ on some Banach space $X$ is called \emph{sectorial of angle $\phi$} if the spectrum $\sigma(A)$ of $A$ is contained in $\overline{\Sigma_{\phi}}$ and
		\begin{equation*}
			\sup_{\lambda \not\in \overline{\Sigma_{\phi}}} (\abs{\lambda} + 1) \norm{R(\lambda,A)} < \infty.
		\end{equation*}
		A family of linear operators $A_i\colon D(A_i) \to X$ for $i \in I$ is \emph{uniformly sectorial} if $\sigma(A_i) \subset \overline{\Sigma_{\phi}}$ for all $i \in I$ and if there exists $C > 0$ with
		\begin{equation*}
			\sup_{\lambda \not\in \overline{\Sigma_{\phi}}} (\abs{\lambda} + 1) \norm{R(\lambda,A_i)} \le C \qquad \text{for all } i \in I.
		\end{equation*}
	\end{definition}
	
	Recall that a closed operator $A$ is sectorial if and only if $-A$ generates an exponentially stable analytic semigroup. In particular, $A^{-1}$ is invertible.
	
	In the following we need interpolation and extrapolation spaces associated to a sectorial operator $A$ on some Banach space, a fully developed theory carefully presented in~\cite{Ama95}. We only discuss spaces associated to the complex interpolation method $[\cdot,\cdot]_{\theta}$. The results to be obtained hold for several other, but not all, scales of interpolation and extrapolation spaces. As a unified treatment would lead to a more abstract presentation, we focus on one particular setting.
	
	We define $X_{1,A} = D(A)$ endowed with the norm $x \mapsto \norm{Ax}$ and $X_{-1,A}$ as the completion of $X$ with respect to the norm $x \mapsto A^{-1}x$. For $\theta \in (0,1)$ we further let $X_{\theta,A} = [X, X_{1,A}]_{\theta}$ and $X_{-\theta,A} = [X,X_{-1,A}]_{\theta}$. The operator $A\colon X_{1,A} \to X$ and its extension $A_{-1}\colon X \to X_{-1,A}$ are isometries. By interpolation, for $\theta \in (0,1)$ the part $A_{-\theta}$ of $A$ in $X_{-\theta,A}$ is an isometry $A_{-\theta}\colon X_{1-\theta,A} \to X_{-\theta,A}$. The operator $A_{-1}$ is sectorial on $X_{-1,A}$ with $\rho(A_{-1}) = \rho(A)$ and satisfies the same sectorial estimates as $A$. By interpolation, the same holds for the operators $A_{-\theta}$ on $X_{-\theta,A}$. Considering duality, if $X$ is reflexive, one has $(X_{\theta,A})' \simeq X'_{-\theta,A}$ and $(A_{\theta})' = A'_{-\theta}$ with respect to the pairing induced by $\langle \cdot, \cdot, \rangle_{X,X'}$. Extrapolation spaces allow us to define a weaker notion of solution for~\eqref{eq:nacp}. 
	
	\begin{proposition}\label{prop:fundamental_identity}
		Let $(A(t))_{t \in [0,T]}$ for $T > 0$ be uniformly sectorial operators on some Banach space $X$. If $u$ is a maximal $L^p$-regularity solution of~\eqref{eq:nacp} for the initial value $u_0 = 0$, then for every fixed $t \in [0,T]$ one has in $X_{-1,A(t)}$
		\begin{equation}
			\label{eq:splitting}
			\begin{split}
    			u(t) & = \int_{0}^t e^{-(t-s)A_{-1}(t)} (A_{-1}(t) - A(s))u(s) \d s + \int_{0}^t e^{-(t-s)A(t)} f(s) \d s \\
    			& \eqqcolon \int_{0}^t K_1(t,s)u(s) \d s + \int_{0}^t K_2(t,s) f(s) \d s \eqqcolon (S_1u)(t) + (S_2f)(t). %
			\end{split}
		\end{equation}
	\end{proposition}
	\begin{proof}
		Fix $t \in (0,T)$. Consider $v\colon [0,t] \to X$ given by $v(s) = e^{-(t-s)A(t)} u(s)$. Then $v$ is differentiable in $X$ almost everywhere and for almost every $s \in (0,t)$ we have
		\begin{equation*}
			\begin{split}
				\dot{v}(s) & = A(t)e^{-(t-s)A(t)}u(s) + e^{-(t-s)A(t)} \dot{u}(s) \\
				& = e^{-(t-s)A_{-1}(t)} (A_{-1}(t) - A(s))u(s) + e^{-(t-s)A(t)} f(s).
			\end{split}
		\end{equation*}
		Notice that $(A_{-1}(t) - A(s))u(s) \in X_{-1,A(t)}$ for almost every $s \in (0,T)$. Hence, integrating both sides in $X_{-1,A(t)}$, we get by the fundamental theorem of calculus
		\begin{equation*}
			v(t) = v(0) + \int_{0}^t \dot{v}(s) \d s.
		\end{equation*}
		Inserting the explicit terms for $v$ and $\dot{v}$ and using $u_0 = 0$ yields equation~\eqref{eq:splitting}.
	\end{proof}
	
\section{Existence and Uniqueness of Integrated Solutions}

	In this section we show that under certain assumptions a unique solution of~\eqref{eq:splitting} exists. The crucial assumption we make from now on is that on a certain extrapolation space the domains of the operators get independent. For concrete differential operators endowed with some boundary condition this is usually satisfied.

	\begin{definition}
    	For $\theta \in [0,1]$ a family $(A(t))_{t \in [0,T]}$ of sectorial operators on some Banach space $X$ is called \emph{$\theta$-stable} if there exists a Banach space $X_{\theta,A}$ and $K \ge 0$ such that for all $t \in [0,T]$ the spaces $X_{\theta,A(t)}$ and $X_{\theta,A}$ agree as vector spaces and
    	\begin{equation}
			\label{eq:theta_stable}
    		K^{-1} \norm{x}_{\theta,A} \le \norm{x}_{\theta,A(t)} \le K \norm{x}_{\theta,A} \qquad \text{for all } x \in X_{\theta,A}
    	\end{equation}
		and if the same also holds for some space $X_{\theta-1,A}$ and all spaces $X_{\theta-1,A(t)}$.
	\end{definition}
	
	Note that $(A(t))_{t \in [0,T]}$ is $1$-stable if and only if the domains $D(A(t))$ agree for all $t \in [0,T]$ and are uniformly equivalent. %
	Since we will be faced with operator-valued singular operators, we rely on tools from vector-valued harmonic analysis. It is by now well-understood that the classical multiplier results only hold in the vector-valued setting if one makes additional assumptions both on the Banach space and the multiplier. This leads to UMD spaces and the concept of $\mathcal{R}$-boundedness. 
	
	\begin{definition}
		A Banach space $X$ is called a \emph{UMD space} if for one or by Hörmander's condition all $p \in (1, \infty)$ the vector-valued Hilbert transform
		\begin{equation*}
			(Hf)(x) = \lim_{\epsilon \downarrow 0} \int_{\abs{t} \ge \epsilon} \frac{f(x-t)}{t} \d t
		\end{equation*}
		initially defined on $C^{\infty}(\IR^n;X)$ extends to a bounded operator $L^p(\IR;X) \to L^p(\IR;X)$.
	\end{definition}
	
	For our purposes it is sufficient to know that Hilbert and $L^p$-spaces for $p \in (1, \infty)$ are UMD spaces. For detailed information on UMD spaces we refer to~\cite{Fra86} and~\cite{Bur01}, whereas more on $\mathcal{R}$-boundedness can be found in~\cite{DHP03} and~\cite{KunWei04}.
	
	\begin{definition}
		Let $X$ and $Y$ be Banach spaces. A subset $\mathcal{T} \subseteq \mathcal{B}(X,Y)$ is called \emph{$\mathcal{R}$-bounded} if there exists a constant $C \ge 0$ such that for all $n \in \IN$, $T_1, \ldots, T_n \in \mathcal{T}$, $x_1, \ldots, x_n$ and all independent identically distributed random variables $\epsilon_1, \ldots, \epsilon_n \in X$ on some probability space $(\Omega, \Sigma, \IE)$ with $\IP(\epsilon_k = \pm 1) = \frac{1}{2}$ one has
		\begin{equation*}
			\IE \biggnorm{\sum_{k=1}^n \epsilon_k T_k x_k}_Y \le C \IE \biggnorm{\sum_{k=1}^n \epsilon_k x_k}_X.
		\end{equation*}
		The smallest constant $C \ge 0$ for which this holds is denoted by $\mathcal{R}(\mathcal{T})$. Further, we define $\Rad X$ as the closure of all finite sums of the form $\sum_{k=1}^n \epsilon_k x_k$ in $L^1(\Omega,\Sigma,\IE)$.
	\end{definition}
	
	We write $\mathcal{R}^{X \to Y}$ to indicate between which spaces the mapping is considered if it is not clear from the context. Every $\mathcal{R}$-bounded set is bounded in $\mathcal{B}(X,Y)$. If both $X = Y$ are Hilbert spaces, then the converse holds as well. Further, Kahane's contraction principle sates that $\{ z \Id: \abs{z} \le 1 \}$ has $\mathcal{R}$-bound at most $2$ on every Banach space. By a celebrated theorem of Weis~\cite{Wei01}, on a UMD space the autonomous problem $A(t) = A$ has maximal $L^p$-regularity for one are all $p \in (1, \infty)$ if and only if $A$ is an $\mathcal{R}$-sectorial operator up to shifts.
	
	\begin{definition}\label{def:r_sectorial}
		A linear operator $A\colon D(A) \to X$ on some Banach space $X$ is called \emph{$\mathcal{R}$-sectorial of angle $\phi$} if the spectrum $\sigma(A)$ of $A$ is contained in $\overline{\Sigma_{\phi}}$ and
		\begin{equation*}
			\R{ (\abs{\lambda} + 1) R(\lambda,A): \lambda \not\in \overline{\Sigma_{\phi}}} < \infty.
		\end{equation*}
		A family of linear operators $A_i\colon D(A_i) \to X$ for $i \in I$ is \emph{uniformly $\mathcal{R}$-sectorial} if $\sigma(A_i) \subset \overline{\Sigma_{\phi}}$ for all $i \in I$ and if there exists $C > 0$ with
		\begin{equation*}
			\R{ (\abs{\lambda} + 1) R(\lambda,A_i): \lambda \not\in \overline{\Sigma_{\phi}}} \le C \qquad \text{for all } i \in I.
		\end{equation*}
	\end{definition}
	
	By interpolation one obtains corresponding $\mathcal{R}$-boundedness estimates on the induced extrapolation spaces. The following result is not new~\cite[Lemma~6.9]{HHK06}, we give a proof for the sake of completeness. For its proof we use the fact that for an interpolation couple $(X,Y)$ of UMD spaces we have~\cite[Proposition~3.14]{KaiSaa12}
	\begin{equation}
		\label{eq:rad_interpolation}
		[\Rad(X),\Rad(Y)]_{\theta} = \Rad([X,Y]_{\theta})
	\end{equation}
	
	\begin{lemma}\label{lem:fractional_r}
		Let $A\colon D(A) \to X$ be an $\mathcal{R}$-sectorial operator on a UMD space $X$. Then for all $\theta_2 > \theta_1 \in [-1,1]$ with $\theta_2 - \theta_1 \le 1$ one has with $\phi$ as in Definition~\ref{def:r_sectorial} and with constants independently of the choice of $A$
		\begin{align*}
			\MoveEqLeft \mathcal{R}^{X_{\theta_1,A} \to X_{\theta_2,A}} \{ (1+\abs{\lambda})^{1-(\theta_2 - \theta_1)} R(\lambda,A): \lambda \not\in \overline{\Sigma_{\phi}} \} \\
			& \lesssim \R{ (\abs{\lambda} + 1) \norm{R(\lambda,A)}: \lambda \not\in \overline{\Sigma_{\phi}}}.
		\end{align*}
	\end{lemma}
	\begin{proof}
		The assertion holds for $\theta_1 = \theta_2 \in \{-1,1\}$. By complex interpolation and~\eqref{eq:rad_interpolation} this extends to $\theta_1 = \theta_2 \in [-1,1]$. Since $AR(\lambda,A) = \lambda R(\lambda,A) - \Id$, one has for all $\theta_1 \in [-1,0]$
		\begin{equation*}
			\mathcal{R}^{X_{\theta_1,A} \to X_{\theta_1+1,A}} \{ R(\lambda,A): \lambda \not\in \overline{\Sigma_{\phi}} \} < \infty.
		\end{equation*}
		For the case of general $\theta_2$ with $\theta_2 - \theta_1 \le 1$ consider for given $n \in \IN$, $\lambda_1, \ldots, \lambda_n \not\in -\overline{\Sigma_{\phi}}$ and $x_1, \ldots, x_n \in X$ the mapping $\mathcal{S} = \{ z \in \IC: \Re z \in [0,1] \} \to \Rad(X_{\theta_1,A}) + \Rad(X_{\theta_1+1,A})$ given by
		\begin{equation*}
			\mathcal{T}_z\colon \sum_{k=1}^{n} \epsilon_k x_k \mapsto \sum_{k=1}^{n} \epsilon_k (1 + \lambda_k)^{z} R(\lambda_k,-A)x_k.
		\end{equation*}
		The mapping $z \mapsto \mathcal{T}_z$ is continuous on $\mathcal{S}$ and analytic in the interior of $\mathcal{S}$ and it follows from Kahane's contraction principle that the norms of $\mathcal{T}_{it}$ and $\mathcal{T}_{1+it}$ as operators in $\mathcal{B}(\Rad(X_{\theta_1,A}), \Rad(X_{\theta_1,A}))$ and $\mathcal{B}(\Rad(X_{\theta_1,A}), \Rad(X_{\theta_1+1,A}))$ are bounded by $e^{\abs{t} \phi}$ up to a uniform constant. Hence, it follows from the generalized Stein interpolation theorem~\cite{Voi92} and~\eqref{eq:rad_interpolation} that
		\begin{equation*}
			\mathcal{T}_{\alpha}\colon \Rad(X_{\theta_1,A}) \to \Rad(X_{\theta_1+\alpha,A}). \qedhere
		\end{equation*}
	\end{proof}
	
	\begin{remark}
		Curiously, the above result fails for the negative Laplacian and the real interpolation method~\cite[Example~6.13]{HHK06}. Hence, this is one step where one cannot work with arbitrary extrapolation spaces.
	\end{remark}
	
	We establish the existence of a unique solution of~\eqref{eq:splitting} assuming Hölder regularity.
	
	\begin{definition}
		A function $f\colon [0,T] \to X$ with values in some Banach space $X$ is \emph{$\alpha$-Hölder continuous} for $\alpha \in (0,1]$ if $\norm{f(t) - f(s)} \le C \abs{t-s}^{\alpha}$ for some $C \ge 0$ and all $t,s \in [0,T]$. We denote by $C^{\alpha}([0,T];X)$ the space of all such functions.
	\end{definition}
	
	We are now ready to prove the existence of integrated solutions.
	
	\begin{proposition}\label{prop:wnacp}
		For $T > 0$ and $\theta \in (0,1]$ let $(A(t))_{t \in [0,T]}$ be a $\theta$-stable family of uniformly $\mathcal{R}$-sectorial operators on some UMD space $X$. Suppose there exist $\alpha \in (0,1]$ with $A_{-1} \in C^{\alpha}([0,T];\mathcal{B}(X_{\theta,A}, X_{\theta-1,A}))$. Then for all $p \in (1, \infty)$ and $f \in L^p([0,T];X)$ there exists a unique solution $u$ of the integral equation~\eqref{eq:splitting} in $L^p([0,T];X_{\theta,A})$. Further, one has $u \in W^{1,p}([0,T];X_{\theta-1,A}) \cap L^p([0,T];X_{\theta,A})$ and
    	\begin{equation*}
    		\LeftEqNo
    		\label{eq:wnacp}\tag{WNACP}
    		\left\{
    		\begin{aligned}
    			\dot{u}(t) + A_{\theta-1}(t)u(t) & = f(t) \\
    			u(0) & = 0.
    		\end{aligned}
    		\right.
    	\end{equation*}
	\end{proposition}
	\begin{proof}
		First note that by the uniform sectorial estimates and the properties of extrapolation spaces we have the uniform estimate
		\begin{align*}
			\normalnorm{e^{-(t-s)A_{-1}(t)}}_{\mathcal{B}(X_{\theta-1,A},X_{\theta,A})} \lesssim \abs{t-s}^{-1}.
		\end{align*}
		Using this together with the assumed Hölder regularity on $A_{-1}(\cdot)$ we get
		\begin{equation}
			\label{eq:k_1}
			\norm{K_1(t,s)}_{\mathcal{B}(X_{\theta,A},X_{\theta,A})} \lesssim \abs{t-s}^{\alpha - 1}.
		\end{equation}
		By Young's inequality for convolutions we then have the norm estimate
		\begin{align*}
			\norm{S_1 u}_{L^p([0,T];X_{\theta, A})} \le \int_0^T s^{\alpha - 1} \d s \norm{u}_{L^p([0,T];X_{\theta,A})} = \alpha^{-1} T^{\alpha} \norm{u}_{L^p([0,T];X_{\theta,A})}.
		\end{align*}
		Let us show the uniqueness of solutions of~\eqref{eq:splitting} in $L^p([0,T];X_{\theta,A})$. Since the equation is linear, it suffices to consider a solution with $u = S_1 u$. Now, for sufficiently small $T_0$ we have $\norm{S_1} < 1$. Hence, $\Id - S_1$ is invertible and consequently $u_{|[0,T_0]} = 0$. Using this information we see that~\eqref{eq:splitting} for $t > T_0$ reduces to
		\begin{equation*}
			u(t) = \int_{T_0}^T e^{-(t-s)A_{-1}(t)} (A_{-1}(t) - A_{-1}(s))u(s) \d s.
		\end{equation*}
		By the same argument as before we see that the operator defined by the left hand side is bounded on $L^p([T_0,2T_0];X_{\theta,A)}$. Hence, $u_{|[T_0,2T_0]} = 0$. Iterating this argument finitely many timesgives $u = 0$.
				
		Further, observe that for the kernel of $S_2$ one has by Lemma~\ref{lem:fractional_r} the estimate
		\begin{equation}
			\label{eq:k_2}
			\norm{K_2(t,s)}_{\mathcal{B}(X, X_{\theta,A(t)})} = \normalnorm{e^{-(t-s)A(t)}}_{\mathcal{B}(X, X_{\theta,A(t)})} \lesssim \abs{t-s}^{-\theta}.
		\end{equation}
		It follows from Young's inequality that $S_2 \in \mathcal{B}(L^p([0,T];X),L^p([0,T];{X_{\theta,A}}))$. 
		Since $t \mapsto A_{\theta-1} \in \mathcal{B}(X_{\theta,A}, X_{\theta-1,A})$ is a fortiori continuous, it follows from perturbation arguments and Lemma~\ref{lem:fractional_r} that~\eqref{eq:wnacp} has non-autonomous maximal $L^p$-regularity for all $p \in (1, \infty)$, see~\cite[Theorem~2.5]{PruSch01} or~\cite[Theorem~2.7]{ACFP07}. Hence, there exists a unique $w \in W^{1,p}([0,T];X_{\theta-1,A}) \cap L^p([0,T];X_{\theta,A})$ satisfying~\eqref{eq:wnacp}. Using the same argument as in Proposition~\ref{prop:fundamental_identity} we see that $w$ satisfies~\eqref{eq:splitting}. By the uniqueness shown in the first part we have $w = u$.
	\end{proof}
	
	\section{Bootstrapping Regularity}
	
	Assuming Hölder regularity, we now improve the regularity of integrated solutions.
		
	\begin{proposition}\label{prop:bootstrapping}
		For $T > 0$ and $\theta \in (0,1]$ let $(A(t))_{t \in [0,T]}$ be a $\theta$-stable family of uniformly sectorial operators on some Banach space $X$ satisfying $A \in C^{\alpha}([0,T];\mathcal{B}(X_{\theta,A}, X_{\theta-1,A}))$ for some $\alpha \in (0,1]$. 
    	If either
		\begin{thm_enum}
			\item $p \in (\frac{1}{1-\theta},\infty)$ and $q \in (1, \infty]$, or
			\item $p = \frac{1}{1-\theta}$ and $q \in (1,\infty)$, or
			\item $p \in (1, \frac{1}{1-\theta})$ and $q \in (1, \frac{p}{1-p(1-\theta)}]$,
		\end{thm_enum}
		then there exists $C_{pq} > 0$ depending only on $T$, $K$ in~\eqref{eq:theta_stable} and the constants of the sectorial and Höler estimates such that for all solutions $u \in L^p([0,T];X_{\theta,A})$ of~\eqref{eq:splitting}
    	\begin{equation*}
    		\norm{u}_{L^q([0,T];X_{\theta,A})} \le C_{pq} \norm{u}_{L^p([0,T];X_{\theta,A})}.
    	\end{equation*} 
	\end{proposition}
	\begin{proof}
		By Young's inequality for convolutions and the kernel estimate~\eqref{eq:k_1} we have for $q, p, r \in (1, \infty)$ with $\frac{1}{r} + \frac{1}{p} = 1 + \frac{1}{q}$ the estimate
		\begin{align*}
			\MoveEqLeft \biggl( \int_0^T \normalnorm{(S_1 u)(t)}_{X_{\theta,A}}^q \d t \biggr)^{1/q} \le \biggl( \int_0^T \biggl( \int_0^t (t-s)^{\alpha-1} \normalnorm{u(s)}_{X_{\theta,A}} \d s \biggr)^q \d t \biggr)^{1/q} \\
			& \lesssim \normalnorm{s \mapsto s^{\alpha-1}}_{L^{r,\infty}} \biggl( \int_0^T \norm{u(s)}_{X_{\theta,A}}^p \d s \biggr)^{1/q}.  
		\end{align*}
		The weak $L^r$ norm is finite for $r \in (1, \frac{1}{1-\alpha}]$. Hence, $S_1$ is a bounded operator $L^p([0,T];X_{\theta,A}) \to L^q([0,T];X_{\theta,A})$ for all $p < \frac{1}{\alpha}$ and $q \in [1, \frac{p}{1-p\alpha}]$. If $p > \frac{1}{\alpha}$, then 
		\begin{align*}
			\MoveEqLeft \norm{(S_1 u)(t)}_{X_{\theta,A}} \le \biggl( \int_0^t \norm{K_1(t,s)}^{p'} \d s \biggr)^{1/p'} \biggl( \int_0^t \norm{u(s)}_{X_{\theta,A}} \d s \biggr)^{1/p} \\
			& = \biggl( \int_0^t \abs{t-s}^{p'(\alpha-1)} \d s \biggr)^{1/p'} \norm{u}_{L^p([0,T];X_{\theta,A})}.
		\end{align*}
		Hence, $S_1\colon L^p([0,T];X_{\theta,A}) \to L^{\infty}([0,T];X_{\theta,A})$ is bounded for $p > \frac{1}{\alpha}$. One can also use Young's inequality together with the kernel estimate~\eqref{eq:k_2} to obtain a similiar estimate for $S_2$. Namely, for $p, p, r \in (1, \infty)$ with $\frac{1}{r} + \frac{1}{p} = 1 + \frac{1}{q}$ we have
		\begin{align*}
			\MoveEqLeft \biggl( \int_0^T \norm{(S_2f)(t)}^{q}_{X_{\theta,A}} \d t \biggr)^{1/q} \le \biggl( \int_0^T \biggr( \int_0^t (t-s)^{-\theta} \norm{f(s)}_{X_{\theta,A}} \d s \biggr)^{q} \d t \biggr)^{1/q} \\
			& \lesssim \norm{s \mapsto s^{-\theta}}_{L^{r,\infty}} \biggr( \int_0^T \norm{f(s)}^p_{X_{\theta}} \d s \biggr)^{1/q}.
		\end{align*}
		This time the $L^{r,\infty}$-norm is finite for $r \in (1, \theta^{-1}]$. Hence, $S_2\colon L^p([0,T];X_{\theta,A}) \to L^q([0,T];X_{\theta,A})$ is bounded for all $p < \frac{1}{1 - \theta}$ and $q \in [1, \frac{p}{1-p(1-\theta)}]$. Further, one has $S_2\colon L^p([0,T];X_{\theta,A}) \to L^{\infty}([0,T];X_{\theta,A})$ for $p > \frac{1}{1 - \theta}$. For the stated result, we iterate the regularity improvement to bootstrap the regularity of $u$.
	\end{proof}
		
\section{Maximal Regularity Results under Fractional Sobolev Regularity}

	In this section we come to the heart of the proof. We need to establish the boundedness of $A(\cdot) S_2\colon L^p([0,T];X) \to L^p([0,T];X)$ which requires some preliminary work. We rely on the following Hölder continuity of the $\mathcal{R}$-boundedness constant.
	
	\begin{lemma}\label{lem:r_boundedness}
		For $\theta \in (0,1]$ let $(A(t))_{t \in \IR}$ be a $\theta$-stable family of uniformly $\mathcal{R}$-sectorial operators on some UMD space $X$. Suppose there exists $\alpha \in (0,1]$ with $A_{-1} \in C^{\alpha}([0,T];\mathcal{B}(X_{\theta,A}, X_{\theta-1,A}))$. 
		Then for all $k \in \IN_0$ there exists a constant $C_k > 0$ depending only on $K$ in~\eqref{eq:theta_stable} and the constants in the Hölder and $\mathcal{R}$-sectorial estimate of Definition~\ref{def:r_sectorial} such that for all $t,h \in \IR$
		\begin{equation*}
			\R{ (1+\abs{\xi})^k \left(\frac{\partial}{\partial \xi} \right)^k \biggl[ i\xi (R(i\xi,A(t+h)) - R(i\xi,A(t))) \biggr]: \xi \in \IR} \le C_k \abs{h}^{\alpha}.
		\end{equation*}
	\end{lemma}
	\begin{proof}
		We first establish the case $k = 0$. For all $t,h \in \IR$ the resolvent identity gives
		\begin{equation*}
			R(i\xi, A(t+h)) - R(i\xi, A(t)) = R(\xi, A_{-1}(t+h)) [A_{-1}(t) - A_{-1}(t+h)] R(i \xi, A(t)).
		\end{equation*}
		By the assumed Hölder regularity on $A_{-1}$ and Lemma~\ref{lem:fractional_r} we get for all $t, h \in \IR$
		\begin{align*}
			\MoveEqLeft \mathcal{R}^{X \to X}\{ i\xi (R(i\xi, A(t+h)) - R(i\xi, A(t))) \} \\
			& \lesssim \mathcal{R}^{X_{\theta-1,A} \to X} \{ (1+\abs{\xi})^{1-\theta} R(\xi,A_{-1}(t+h)) \} \\
			& \cdot \norm{A_{-1}(t) - A_{-1}(s)}_{\mathcal{B}(X_{\theta,A}, X_{\theta-1,A})} \cdot \mathcal{R}^{X \to X_{\theta,A}} \{ (1+\abs{\xi})^{\theta} R(i\xi,A(t)) \} \lesssim \abs{h}^{\alpha}.
		\end{align*}
		For the case $k \ge 1$ notice that the map $S\colon z \mapsto R(z,A(t+h)) - R(z,A(t)) \in \mathcal{B}(X)$ is analytic on the complement of some shifted sector $\Sigma_{\phi} + \epsilon$ and that the above estimate holds there by the same argument. It follows from the Cauchy integral representation of derivatives~\cite[Example~2.16]{KunWei04} that for $S(z) = z (R(z, A(t+h)) - R(z, A(t)))$ 
		\begin{equation*}
			\mathcal{R} \biggl\{ (1+\abs{z})^k \left( \frac{d}{d z} \right)^{k} S(z): z \not\in \Sigma_{\phi} \biggr\} \lesssim \mathcal{R} \biggl\{ S \biggl(i\xi + \frac{\epsilon}{2} \biggr): \xi \in \IR \biggr\} \lesssim \abs{h}^{\alpha}. \qedhere
		\end{equation*}
	\end{proof}
	
	\begin{proposition}\label{prop:s2_bounded}
		For $T > 0$ and $\theta \in (0,1]$ let $(A(t))_{t \in \IR}$ be a $\theta$-stable family of uniformly $\mathcal{R}$-sectorial operators on some UMD space $X$. Suppose there exist $\alpha \in (0,1]$ with $A \in C^{\alpha}([0,T];\mathcal{B}(X_{\theta,A}, X_{\theta-1,A}))$.
		Then $S_2\colon L^p([0,T];X) \to L^p([0,T];X)$ is bounded for all $p \in (1, \infty)$ and its norm only depends on $p$, $K$ in~\eqref{eq:theta_stable} and the constants in the Hölder and $\mathcal{R}$-sectorial estimates.
	\end{proposition}
	\begin{proof}
		It is shown in~\cite[p.~1053]{HieMon00b} or~\cite[Section~2.4.1]{Fac15c} that the boundedness of $S_2$ follows from the boundedness of the pseudodifferential operator
		\begin{equation*}
			(\hat{S}f)(t) = \int_{-\infty}^{\infty} a(t,\xi) \hat{f}(\xi) e^{2\pi i t \xi} \d\xi
		\end{equation*}
		for the operator-valued symbol $a\colon \IR \times \IR \to \mathcal{B}(X)$ given by
		\begin{equation*}
			a(t,\xi) = \begin{cases}
				i\xi R(i\xi, A(0)), & t < 0 \\
				i\xi R(i\xi, A(t)), & t \in [0,T] \\
				i\xi R(i\xi, A(T)), & t > T.
			\end{cases}
		\end{equation*}
		Such operators are well-studied and understood. Applying~\cite[Theorem~17]{HytPor08} and~\cite[Remark~20]{HytPor08} (the dependence on the constants is not explicitly stated) in the one-dimensional and one-parameter case, we see that $S\colon L^p([0,T];X) \to L^p([0,T];X)$ is bounded for all $p \in (1, \infty)$ provided
		\begin{equation*}
			\R{ (1+\abs{\xi})^k \left(\frac{\partial}{\partial \xi} \right)^k \bigl[ a(t+h,\xi) - a(t,\xi) \bigr]: \xi \in \IR} \lesssim \abs{h}^{\alpha}
		\end{equation*}
		holds for some $\alpha \in (0,1]$ and all $k = 0, 1, 2$. This has been verified in Lemma~\ref{lem:r_boundedness}.
	\end{proof}

	The next proposition shows that in many cases it is sufficient to show maximal $L^p$-regularity for initial value zero. This is well-known in the autonomous case. The arguments have been used before, see for example~\cite[Theorem~6.2]{DieZac16}.
	
	\begin{proposition}\label{prop:initial_values}
		Let $X$ be a Banach space, $p \in (1, \infty)$ and suppose one has for all $T > 0$ maximal $L^p$-regularity results for classes $\mathcal{C}_T$ of non-autonomous sectorial operators on $X$ and $[0,T]$ with initial value $u_0 = 0$. Suppose that for $(A(t))_{t \in [0,T]}$ in $\mathcal{C}_T$ and $0 < T_1 \le T \le T_2$ also the non-autonomous operator $(B(t))_{t \in [0,2T+T_2-T_1]}$
		\begin{equation*}
			B(t) = \begin{cases}
				A(0) & \text{for } t \in [0,T], \\
				A(t-T+T_1) & \text{for } t \in [T,T+T_2-T_1], \\
				A(T_2) & \text{for } t \in [T+T_2-T_1,2T+T_2-T_1]
			\end{cases}
		\end{equation*}
		is in $\mathcal{C}_{2T+T_2-T_1}$. Then $(A(t))_{t \in [0,T]}$ has maximal $L^p$-regularity for all initial values $u_0 \in (D(A(0)),X)_{\frac{1}{p},p}$ and further $u(t) \in (D(A(t)),X)_{\frac{1}{p},p}$ for all $t \in [0,T]$. Further, the maximal regularity estimate only involves a possible new dependence and $T$.
	\end{proposition}
	\begin{proof}		
		We first deal with the initial values. Here we set $T_1 = 0$ and $T_2 = T$. By the characterization of real interpolation spaces via the trace method~\cite[Proposition~1.2.10]{Lun95} and a cut-off argument there is some $C_T > 0$ such that for $u_0 \in (D(A(0)), X)_{\frac{1}{p},p}$ there exists $v \in W^{1,p}([0,T];X) \cap L^p([0,T];D(A(0))$ with $v(0) = 0$, $v(T) = u_0$ and
		\begin{equation*}
			\norm{A(0)v}_{L^p([0,T];X)} + \norm{\dot{v}}_{L^p([0,T];X)} \le C_T \norm{u_0}_{(D(A(0)), X)_{\frac{1}{p},p}}.
		\end{equation*}
		We define $g \in L^p([0,3T];X)$ as
		\begin{equation*}
			g(t) = \begin{cases}
				\dot{v}(t) + A(0)v(t) & \text{for } t \in [0,T), \\
				f(t - T) & \text{for } t \in [T,2T], \\
				0 & \text{for } t \in (2T,3T].
			\end{cases}
		\end{equation*}
		Note that by assumption $(B(t))_{t \in [0,3T]}$ lies in $\mathcal{C}_{3T}$ and therefore has maximal $L^p$-regularity for $u_0 = 0$. We denote the unique solution of~\eqref{eq:nacp} for $(B(t))_{t \in [0,3T]}$ by $w$. By the uniqueness of solutions in the autonomous case we have $w = v$ on $[0,T]$. In particular, we have $w(T) = v(T) = u_0$. As a consequence we see that $u(t) = w(t+T)$ solves~\eqref{eq:nacp} for $u(0) = w(T) = u_0$ as desired. Further,
		\begin{align*}
			\MoveEqLeft \norm{A(\cdot)u(\cdot)}_{L^p([0,T];X)} + \norm{u}_{W^{1,p}([0,T];X)} \\
			& \lesssim \norm{g}_{L^p([0,T];X)} \lesssim \norm{f}_{L^p([0,T];X)} + C_T \norm{u_0}_{(D(A(0)), X)_{\frac{1}{p},p}}.
		\end{align*}
		In a similiar fashion for $t \in (0,T]$ we choose $T_1 = 0$ and $T_2 = t$. The solution $z$ of the corresponding problem for $g$ agrees with the solution $w$ of the first part on $[0,T+t]$ and solves the autonomous problem $\dot{z}(s) + A(s)z(s) = g(s)$ on $[T+t,2T+t]$. Since solutions of the autonomous problem take values in the corresponding trace spaces~\cite[Theorem III.4.10.2]{Ama95}, we have $u(t) \in (D(A(t)),X)_{\frac{1}{p},p}$.
	\end{proof}
	
	We now prove maximal regularity under fractional Sobolev regularity.
		
	\begin{definition}
		Let $X$ be a Banach space, $p \in (1, \infty)$ and $\alpha \in (0,1)$. A Bochner measurable function $f\colon [0,T] \to X$ lies in the \emph{homogeneous fractional Sobolev space $\dot{W}^{\alpha,p}([0,T];X)$} provided
		\begin{equation*}
			\norm{f}_{\dot{W}^{\alpha,p}([0,T];X)} = \biggl( \int_0^T \int_0^T \frac{\norm{f(t) - f(s)}^p_X}{\abs{t-s}^{1+\alpha p}} \d s \d t \biggr)^{1/p} < \infty.
		\end{equation*}
		The \emph{inhomogeneous Sobolev space $W^{\alpha,p}([0,T];X)$} is the space of all $f \in L^p([0,T];X)$ with $\norm{f}_{\dot{W}^{\alpha,p}([0,T];X)} < \infty$.
	\end{definition}
	
	We remark that there exist equivalent definitions of fractional Sobolev spaces based on Littlewood--Paley decompositions~\cite[Section~3, (3.5)]{Ama00}. The usual embedding results for Sobolev spaces into Hölder spaces hold: for $\alpha \in (0,1)$ and $p \in (1, \infty)$ with $\alpha > \frac{1}{p}$ one has $W^{\alpha,p}([0,T];X) \hookrightarrow C^{\alpha - 1/p}([0,T];X)$~\cite[Corollary~26]{Sim90}. We are now ready to formulate and prove our general maximal regularity result.
	
	\begin{theorem}\label{thm:mr}
		For $T > 0$ and $\theta \in (0,1]$ let $(A(t))_{t \in [0,T]}$ be a $\theta$-stable family of uniformly $\mathcal{R}$-sectorial operators on some UMD space $X$ with fractional regularity $A_{-1} \in \dot{W}^{\alpha,q}([0,T];\mathcal{B}(X_{\theta,A}, X_{\theta-1,A}))$.
		Then the non-autonomous problem~\eqref{eq:nacp} has maximal $L^p$-regularity for $p \in (1, \infty)$ if one of the following assumptions holds.
		\begin{thm_enum}
			\item $p \in (1, \frac{1}{1-\theta})$, $q = \frac{1}{1-\theta}$ and $\alpha > 1 - \theta$.
			\item $p \in [\frac{1}{1-\theta}, \infty)$, $q = p$ and $\alpha > 1 - \theta$.
		\end{thm_enum}
		In this case the unique maximal $L^p$-regularity solution $u$ of~\eqref{eq:nacp} satisfies $u(t) \in (D(A(t)),X)_{\frac{1}{p},p}$ for all $t \in [0,T]$ and there exists a constant $C_p > 0$ only depending on $T$, $\alpha$, $\theta$, $K$ in~\eqref{eq:theta_stable}, $\norm{A_{-1}}_{\dot{W}^{\alpha,q}([0,T];\mathcal{B}(X_{\theta,A}, X_{\theta-1,A}))}$ and the constants in the Hölder and $\mathcal{R}$-sectorial estimates with
		\begin{equation*}
			\norm{u}_{W^{1,p}([0,T];X)} + \norm{A(\cdot)u(\cdot)}_{L^p([0,T];X)} \le C (\norm{f}_{{L^p}([0,T];X)} + \norm{u_0}_{(D(A(0)), X)_{\frac{1}{p},p}}).
		\end{equation*}
	\end{theorem}
	\begin{proof}
		Let $u \in W^{1,p}([0,T];X_{\theta,A}) \cap L^p([0,T];X_{\theta-1,A})$ be the unique solution of~\eqref{eq:splitting} given by Proposition~\ref{prop:fundamental_identity}. We show that $u$ has the higher regularity $A_{-1}(t)u(t) \in L^p([0,T];X)$. For this we use the decomposition of $A_{-1}(t)u(t)$ given by the splitting~\eqref{eq:splitting}. 		
		Let us start with the integrability of $A_{-1}(t)(S_1u)(t)$. We will omit subindices in the following estimates. For arbitrary $g \in L^{p'}([0,T];X')$ we have
		\begin{equation}
			\label{eq:testing}
			\begin{split}
    			& \int_0^T \int_0^t \left\langle g(t), A(t)e^{-(t-s)A(t)} (A(t) - A(s))u(s) \right\rangle_{X',X} \d s \d t \\
    			& = \int_0^T \int_0^t \left\langle A'(t) e^{-(t-s)A'(t)}g(t), (A(t) - A(s)) u(s) \right\rangle_{X'_{1-\theta,A'(t)},X_{\theta-1,A(t)}} \d s \d t.
			\end{split}
		\end{equation}
		We now distinguish between the cases $p \in [\frac{1}{1-\theta}, \infty)$, $p = \frac{1}{1-\theta}$ and $p \in (1, \frac{1}{1-\theta}]$. In the first case we know from Proposition~\ref{prop:bootstrapping} that $u \in L^{\infty}([0,T];X_{\theta,A})$. Hence, up to constants \eqref{eq:testing} is dominated by
		\begin{align*}
			\MoveEqLeft \biggl( \int_0^T \int_0^T \frac{\norm{(A(t) - A(s)) u(s)}^p_{X_{\theta-1,A}}}{\abs{t-s}^{1+p\alpha}} \d s \d t \biggr)^{1/p} \\
			& \qquad \cdot \biggl( \int_0^T \int_0^t \normalnorm{A'(t)e^{-(t-s)A'(t)} g(t)}^{p'}_{X_{1-\theta,A'(t)}} \abs{t-s}^{p'(\frac{1}{p} + \alpha)} \d s \d t \biggr)^{1/p'} \\
			& \lesssim \norm{A}_{\dot{W}^{\alpha,p}} \norm{u}_{L^{\infty}([0,T];X_{\theta,A})} \biggl( \int_0^T \int_0^t (t-s)^{p'(\frac{1}{p} + \alpha + \theta - 2)} \d s \norm{g(t)}_X^{p'} \d t \biggr)^{1/p'}.
		\end{align*}
		The inner integral is finite because of the assumption $\alpha > 1 - \theta$. Since $g \in L^{p'}([0,T];X')$ is arbitrary, we get $A_{-1}(\cdot) S_1 u \in L^p([0,T];X)$. The case $p = \frac{1}{1-\theta}$ follows similarly using $L^q([0,T];X_{\theta,A})$ for some big $q$ and the fact that the condition $\alpha > 1 - \theta$ leaves a little room. Let us come to the second case $p \in (1, \frac{1}{1-\theta})$. Here Proposition~\ref{prop:bootstrapping} shows that $u \in L^{p/(1-p(1-\theta))}([0,T];X_{\theta,A})$. Hence, using Hölder's inequality, for~\eqref{eq:testing} and $\beta > 0$ we obtain the estimate
		\begin{align*}
			\MoveEqLeft \biggl( \int_0^T \int_0^T \frac{\norm{A(t) - A(s)}_{\mathcal{B}}^{\frac{1}{1-\theta}}}{\abs{t-s}^{1+ \alpha (1-\theta)^{-1}}} \d s \d t \biggr)^{1-\theta} \biggl( \int_0^T \int_0^t (t-s)^{p'(\alpha + \beta - 1)} \d s \norm{g(t)}_{X'}^{p'} \d t \biggr)^{1/p'} \\
			& \qquad \cdot \biggl( \int_0^T \int_s^T (t-s)^{-\frac{\beta p}{1-p(1-\theta)}} \d t \norm{u(s)}^{\frac{p}{1-p(1-\theta)}} \d s \biggr)^{\frac{1}{p} - (1-\theta)}.
		\end{align*}
		The last integral is finite for $\beta < \theta - \frac{1}{p'}$. Since $\alpha > 1 - \theta$, we can find $\beta \in (0, \theta - \frac{1}{p'})$ for which the second integral is finite as well.
		
		Further, $A_{-1}(\cdot)(S_2f)(\cdot)$ lies in $L^p([0,T];X)$ by Proposition~\ref{prop:s2_bounded}. This shows that the solution satisfies $u(t) \in D(A(t))$ for almost all $t \in [0,T]$ and $A(\cdot)u(\cdot) \in L^p([0,T];X)$. Since $u$ solves~\eqref{eq:wnacp}, it follows that $\dot{u} \in W^{1,p}([0,T];X)$. This finishes the proof in the case $u_0 = 0$. The case of general initial values $u_0 \in (D(A(0)),X)_{\frac{1}{p},p}$ follows from Proposition~\ref{prop:initial_values}. 
	\end{proof}
	
	\begin{remark}
		Compared to the result in~\cite[Theorem~3.3]{Fac16d} we need a weaker $\mathcal{R}$-boundedness result. Further, the time regularity is lowered to some fractional Sobolev regularity at the cost of more regularity on the domain spaces. In order to obtain maximal $L^p$-regularity for all $p \in [(1-\theta)^{-1},\infty)$ our result requires $\mathcal{A} \in \cap_{p \in [(1-\theta)^{-1}, \infty)} \cup_{\epsilon > 0} \dot{W}^{1-\theta+\epsilon,p}([0,T];\mathcal{B}(X_{\theta,A}, X_{\theta-1,A}))$. This is slightly less restrictive than the $\alpha$-Hölder continuity for some $\alpha > (1-\theta)^{-1}$ assumed usually.
	\end{remark}
	
	For non-autonomous problems given by sesquilinear forms on Hilbert spaces one obtains by the same line of thought the following improvement of~\cite{DieZac16}, where only the case $p = 2$ was treated.
	
	\begin{corollary}\label{cor:form}
		Let $V,H$ be Hilbert spaces with dense embedding $V \hookrightarrow H$ and let $a\colon [0,T] \times V \times V \to \IC$ be a non-autonomous bounded coercive sesquilinear form. Then the associated problem~\eqref{eq:nacp} on $H$ has maximal $L^p$-regularity
		\begin{thm_enum}
			\item for $p \in (1,2]$ provided $\mathcal{A} \in \dot{W}^{\frac{1}{2}+\epsilon,2}([0,T];\mathcal{B}(V,V'))$ for some $\epsilon > 0$,
			\item for $p \in [2,\infty)$ provided $\mathcal{A} \in \dot{W}^{\frac{1}{2}+\epsilon,p}([0,T];\mathcal{B}(V,V'))$ for some $\epsilon > 0$.
		\end{thm_enum}
	\end{corollary}
	\begin{proof}
		Repeat the previous proof for $X = H$, $X_{\frac{1}{2},A} = V$ and $X_{-\frac{1}{2},A} = V'$.
	\end{proof}
	
	Note that $V$ and $V'$ only agree with the complex interpolation spaces $X_{\frac{1}{2},A(t)}$ and $X_{-\frac{1}{2},A(t)}$ if the operators $A(t)$ satisfy the so-called Kato square root property. However, this is not necessary to carry out the argument. In the Banach space setting the case $\theta = \frac{1}{2}$ is also of particular interest. We obtain the following corollary relevant for concrete applications (which holds for other values of $\theta$ as well).
	
	\begin{corollary}\label{cor:bip}
		Let $T > 0$ and $(A(t))_{t \in [0,T]}$ be uniformly sectorial operators on a UMD space $X$ such that for $\omega \in (0, \frac{\pi}{2})$ and $M > 0$ the imaginary powers satisfy 
		\begin{equation*}
			\normalnorm{A(t)^{is}} \le M e^{\omega \abs{s}}
		\end{equation*}
		uniformly for all $t \in [0,T]$ and $s \in \IR$. Further, suppose that there exist Banach spaces $X_{\frac{1}{2}}$ and $X_{-\frac{1}{2}}$ such that for all $t \in [0,T]$ the spaces $D(A(t)^{1/2})$ and $D(A(t)^{-1/2})$ agree with $X_{\frac{1}{2}}$ and $X_{-\frac{1}{2}}$ as vector spaces and such that the respective norms are uniformly equivalent for some constant $K > 0$. Then the associated non-autonomous Cauchy problem~\eqref{eq:nacp} has maximal $L^p$-regularity
		\begin{thm_enum}
			\item for $p \in (1,2]$ if $t \mapsto A_{-1}(t) \in \dot{W}^{\frac{1}{2}+\epsilon,2}([0,T];\mathcal{B}(X_{\frac{1}{2}},X_{-\frac{1}{2}}))$ for some $\epsilon > 0$,
			\item for $p \in [2,\infty)$ if $t \mapsto A_{-1}(t) \in \dot{W}^{\frac{1}{2}+\epsilon,p}([0,T];\mathcal{B}(X_{\frac{1}{2}},X_{-\frac{1}{2}}))$ for some $\epsilon > 0$.
		\end{thm_enum}
		The constants in the maximal $L^p$-regularity estimates only depend on $T$, $\epsilon$, $K$, the fractional Sobolev norm of $A_{-1}$ and the constants in the Hölder estimate. 
	\end{corollary}
	\begin{proof}
		Since the operators $A(t)$ have uniformly bounded imaginary powers, it follows from~\cite[Theorem~4.5]{DHP03} that for $\phi \in (\omega, \pi)$
		\begin{equation*}
			\sup_{t \in [0,T]} \R{ \lambda R(\lambda,A(t)): \lambda \not\in \overline{\Sigma_{\phi}}} < \infty
		\end{equation*}
		Since uniformly bounded analytic families are uniformly $\mathcal{R}$-bounded on compact subsets of a common domain~\cite[Proposition~2.6]{Wei01}, the operators $(A(t))_{t \in [0,T]}$ are uniformly $\mathcal{R}$-sectorial. Further, the fractional domains spaces $D(A(t)^{1/2})$ and $D(A(t)^{-1/2})$ are uniformly equivalent to $X_{1/2,A(t)}$ and $X_{-1/2,A(t)}$~\cite[Proposition~2.5]{Fac15c}. As a consequence the family $(A(t))_{t \in [0,T]}$ is $\frac{1}{2}$-stable. This means that we can apply Theorem~\ref{thm:mr}.
	\end{proof}
	
	\begin{remark}
		Corollary~\ref{cor:bip} holds under the slightly weaker assumption that the operators $(A(t))_{t \in [0,T]}$ are uniformly $\mathcal{R}$-sectorial. For this one uses the scale $X_{\theta,A} = D(A^{\theta})$ for $\abs{\theta} \in (0,1)$ and repeats the proof of Theorem~\ref{thm:mr}. The main difference is that one has to use~\cite[Lemma~6.9 (1)]{HHK06} instead of~Lemma~\ref{lem:fractional_r}.
	\end{remark}
	
	\section{Non-Autonomous Maximal Regularity For Elliptic Operators}
	
	In this section we illustrate the consequences of our results to non-autonomous problems governed by elliptic operators in divergence form. We do not present the most general framework here and concentrate on pure second order operators with $\VMO$-coefficients subject to Dirichlet boundary conditions as the used results are already involved and spread over the literature in this special case. However, we give some additional references to the literature. We start with introducing the appropriate function space.
	
	\begin{definition}
		Let $\Omega \subset \IR^n$ be a bounded domain. A bounded measurable function $f\colon \Omega \to \IC$ is of \emph{vanishing mean oscillation} if one has $\inf_{r > 0} \eta_f(r) = 0$ for the modulus
		\begin{equation*}
			\eta_f(r) = \sup_{B: d(B) \le r} \biggl( \frac{1}{\abs{B \cap \Omega}} \int_{B \cap \Omega} \abs{f(x) - f_{B \cap \Omega}}^2 \d x \biggr)^{1/2},
		\end{equation*}
		where $f_{\Omega \cap B}$ denotes the mean of $f$ over $\Omega \cap B$ and the supremum is taken over all balls $B$ of diameter $d(B)$ not exceeding $r$ and centered in $\Omega$.
	\end{definition}
	
	We need the following variant of the Kato square root property on $L^q$-spaces.
	
	\begin{theorem}\label{thm:kato}
		Let $n \in \IN$, $\Omega \subset \IR^n$ a bounded $C^1$-domain, $q \in (1,\infty)$, and $A = (a_{ij})_{1 \le i,j \le n} \in L^{\infty}(\Omega; \IC^{n \times n})$ complex-valued coefficients with
		\begin{equation*}
			\Re \sum_{i,j=1}^n a_{ij}(x) \xi_i \xi_j \ge \delta \abs{\xi}^2 \qquad \text{for all } \xi \in \IR^n
		\end{equation*}
		for some $\delta > 0$ and almost every $x \in \Omega$. Denote by $L_q$ the realization of $-\divv(A\nabla\cdot)$ on $L^q(\Omega)$ subject to Dirichlet boundary conditions. If $a_{ij} \in \VMO(\Omega)$ for all $i,j = 1, \ldots, n$, then there exists $\lambda_0 \ge 0$ such that $L_q + \lambda$ is a sectorial operator on $L^q(\Omega)$ for all $\lambda \ge \lambda_0$ and 		
		\begin{equation*}
			\norm{f}_q + \norm{\nabla f}_q \simeq \normalnorm{(L_q + \lambda)^{1/2} f}_q \qquad \text{for all } f \in W_0^{1,q}(\Omega).
		\end{equation*}
		Moreoever, $\lambda_0$ only depends on $n$, $\Omega$, $q$, $\eta_{a_{ij}}$, $\delta$ and $\norm{A}_{\infty}$. Further, with an additional dependence on $\lambda$ the same holds for the constant in the above equivalence and the sectorial estimates of the operators $L_q + \lambda$.
	\end{theorem}
	\begin{proof}
		Under the made assumptions the operator $L_2$ -- which can be defined for arbitrary complex elliptic coefficients by the method of forms -- satisfies local Gaussian estimates~\cite[Theorem~7]{AusTch01b}. Although not explicitly stated, the coefficients in the estimate only depend on the claimed constants. This has several consequences. First, for $\lambda$ sufficiently large the operator $L_2 + \lambda$ satisfies global Gaussian estimates~\cite[Section~1.4.5, Theorem~18]{AusTch98} and extends to a sectorial operator $L_q + \lambda$ on $L^q(\Omega)$. Secondly, it essentially follows from~\cite[Theorem~4]{AusTch01} that $\normalnorm{(L_q + \lambda)^{1/2}}_q \lesssim \norm{f}_q + \norm{\nabla f}_q$. Here are two additional points to consider. First, the theorem is only stated in the case $\lambda = 0$. The case $\lambda \neq 0$ can be obtained by including terms of lower order in the argument or by arguing as in~\cite[p.~135]{AusTch98}. The second point is -- as always -- the not explicitly stated dependence on the constants. However, taking a close look at the proof in~\cite{AusTch01} one sees that most auxiliary results give the explicit dependence on the constants (on \cite[p.~162]{AusTch01} such a dependence is explicitly stated in a special case). One crucial point needed here is the dependence in the case $p = 2$ which is well-known. This can be found in~\cite[Theorem~1]{AKM06} for a broad class of Lipschitz domains and a combination of~\cite[Theorem~4.2]{EgeHalTol14} and~\cite[Theorems~3.1~\&~3.3 and Section~6]{EgeHalTol16} yields the dependence for general bounded Lipschitz domains.
		
		Now, as in~\cite[p.~135]{AusTch98}, the converse inequality follows if $(L_q + \lambda)^{-1}$ extends to a bounded operator from $W^{-1,q}(\Omega) = (W_0^{1,q'}(\Omega))'$ into $W_0^{1,q}(\Omega)$. Notice that
		\begin{align*}
			\norm{u}_{W^{-1,q}(\Omega)} = \inf \biggl\{ \norm{g}_q + \sum_{k=1}^n \norm{F_k}_q: g, F_k \in L^q(\Omega) \text{ and } \divv F + g = u \biggr\}.  
		\end{align*}
		It is shown in~\cite[Theorem~4]{DonKim10} that for $\lambda \ge 0$ there exists $C \ge 0$ such that for all $F_k,g \in L^q(\Omega)$ there is a unique $u \in W_0^{1,q}(\Omega)$ with $-\divv(A \nabla u) + \lambda u = \divv F + g$ and
		\begin{equation*}
			\norm{u}_{W^{1,q}(\Omega)} \le C \biggl( \norm{g}_q + \sum_{k=1}^n \norm{F_k}_q \biggr).
		\end{equation*} 
		Here, our required dependence on the constants can be read of the lemmata in~\cite[Section~7]{DonKim10}. Note that the above estimate is exactly the boundedness of $(1+L_q)^{-1}\colon W^{-1,q}(\Omega) \to W_0^{1,q}(\Omega)$. This finishes the proof.
	\end{proof}
	
	\begin{remark}
		The estimate $\normalnorm{L^{1/2}f}_p \lesssim \norm{f}_p$ is known under more general assumptions on the coefficients and the domain~\cite[Theorem~4]{AusTch01}. The same holds for the boundedness of $(L_q + \lambda)^{-1}\colon W^{-1,q}(\Omega) \to W_0^{1,q}(\Omega)$ for which originating from~\cite{Kry07b} many results have been obtained in the last years. For a complete list of references we refer to the introduction of~\cite{DonKim16} and for a proof of similar results within the framework of maximal regularity to~\cite{GalVer14} and~\cite{GalVer15}.
	\end{remark}
	
	\begin{theorem}\label{thm:mr_elliptic}
		Let $\Omega \subset \IR^n$ be a bounded $C^1$-domain, $T > 0$ and $a_{ij} \in L^{\infty}([0,T] \times \Omega)$ for $i, j = 1, \ldots, n$. Assume further that the following properties are satisfied.
		\begin{thm_enum}
			\item There exists $\delta > 0$ such that for almost all $(t,x) \in [0,T] \times \Omega$ and all $\xi \in \IR^n$
				\begin{equation*}
					\Re \sum_{i,j=1}^n a_{ij}(t,x) \xi_i \overline{\xi_j} \ge \delta \abs{\xi}^2.
				\end{equation*}
			\item\label{thm:mr_elliptic:vmo} The functions $x \mapsto a_{ij}(t,x)$ lie in $\VMO(\Omega)$ and there is $\eta\colon [0,1] \to [0,\infty]$ with $\lim_{r \downarrow 0} \eta(r) = 0$ and $\eta_{a_{ij}(t,\cdot)} \le \eta$ for all $t \in [0,T]$ and $i,j = 1, \ldots, n$.
		\end{thm_enum}
		For $q \in (1, \infty)$ let $L_q(A(t)) = L_q(t) = -\divv(A(t)\nabla \cdot)$ be the corresponding sectorial operators on $L^q(\Omega)$. Then for all $q \in (1, \infty)$ the non-autonomous problem~\eqref{eq:nacp} associated to $(L_q(t))_{t \in [0,T]}$ has maximal $L^p$-regularity
		\begin{thm_enum}
			\item for $p \in (1,2]$ if $t \mapsto a_{ij}(t,\cdot) \in \dot{W}^{\frac{1}{2}+\epsilon,2}([0,T];L^{\infty}(\Omega))$ for some $\epsilon > 0$,
			\item for $p \in [2,\infty)$ if $t \mapsto a_{ij}(t,\cdot) \in \dot{W}^{\frac{1}{2}+\epsilon,p}([0,T];L^{\infty}(\Omega))$ for some $\epsilon > 0$.
		\end{thm_enum}
		The maximal $L^p$-regularity estimate depends only on $p, q, n, T, \Omega, \delta, \eta, \epsilon,\norm{a_{ij}}_{\infty}$ and the homogenous Sobolev norm in (a) or (b).
	\end{theorem}
	\begin{proof}
		Thanks to the Gaussian estimates discussed in the proof of Theorem~\ref{thm:kato}, the operators $L_q(t) + \lambda$ have uniformly bounded imaginary powers for some $\omega \in (0, \frac{\pi}{2})$ and for sufficiently large $\lambda$. This follows from the general result~\cite[Theorem~4.3]{DuoRob96} (which even gives a bounded $H^{\infty}$-calculus), which does not state the dependence on the constants explicitly.
		Further, it follows from~Theorem~\ref{thm:kato} that $D(L_q(A(t)) + \lambda)^{1/2}) \simeq W_0^{1,q}(\Omega)$ holds uniformly in $t \in [0,T]$. The coefficients $A(t)^T$ satisfy the same assumptions and one therefore has $D(L_{q'}(A(t)^T) + \lambda)^{1/2}) \simeq W_0^{1,q'}(\Omega)$ uniformly as well. For fixed $t \in [0,T]$ let $L_*^{1/2}$ be the adjoint of $(L_{q'}(A(t)^T) + \lambda)^{1/2}$. Then $L_*^{1/2}\colon L^q \to (W_0^{1,q'}(\Omega))' = W^{-1,q}(\Omega)$ extends $(L_q(t)+\lambda)^{1/2}$ and is an isomorphism. Consequently, one has for $u \in L^q(\Omega)$
		\begin{equation*}
			\normalnorm{(L_q(t) + \lambda)^{-1/2} u}_q = \normalnorm{(L_*^{1/2})^{-1} u}_q \simeq \norm{u}_{W^{-1,q}(\Omega)}.
		\end{equation*}
		Hence, $D((L_q(t) + \lambda)^{-1/2}) \simeq W^{-1,q}(\Omega)$ uniformly in $t \in [0,T]$. Therefore $X_{\frac{1}{2}} = W^{1,q}(\Omega)$ and $X_{-\frac{1}{2}} = W^{1,q'}(\Omega)$ in Corollary~\ref{cor:bip}. It remains to check the time regularity. For $u \in W_0^{1,2}(\Omega) \cap W_0^{1,q}(\Omega)$ and $v \in W_0^{1,2}(\Omega) \cap W_0^{1,q'}(\Omega)$ one has
		\begin{align*}
			\MoveEqLeft \normalabs{\langle L_q(t)u - L_q(s)u, v \rangle} = \abs{\int_{\Omega} (A(t) - A(s)) \nabla u \nabla v} \le \norm{A(t)-A(s)}_{\infty} \norm{\nabla u}_q \norm{\nabla v}_{q'}.
		\end{align*}
		By density this extends to all $u \in W_0^{1,q}(\Omega)$ and all $v \in W_0^{1,q'}(\Omega)$. It follows from the assumption that $L_q(\cdot) + \lambda \in \dot{W}^{\alpha,r}([0,T];\mathcal{B}(W_0^{1,q}(\Omega), W^{-1,q}(\Omega)))$ with $\alpha$ and $r$ as in the assumptions. Now, Corollary~\ref{cor:bip} applies and yields the maximal $L^p$-regularity of $(L_q(t) + \lambda)_{t \in [0,T]}$ for $\lambda$ big enough. By a rescaling argument this is equivalent to the maximal $L^p$-regularity of $(L_q(t))_{t \in [0,T]}$.
	\end{proof}
	
\section{Applications to Quasilinear Parabolic Problems}\label{sec:applications}

	In this section we use Theorem~\ref{thm:mr_elliptic} to solve quasilinear parabolic equations.
	\begin{theorem}
		Let $\Omega \subset \IR^n$ be a bounded $C^1$-domain and $T > 0$. For coefficients $A = (a_{ij})\colon \IC \to \IC^n$, $p \in [2, \infty)$, $q \in (1, \infty)$, an inhomogeneous part $f \in L^p([0,T];L^q(\Omega))$ and an initial value $u_0 \in L^q(\Omega)$ satisfying the condition $u_0 \in (D(-\divv A(u_0) \nabla \cdot),L^q(\Omega))_{\frac{1}{p},p}$ consider the problem
    	\begin{equation*}
    		\label{eq:QLP}
    		\tag{QLP}
    		\LeftEqNo
    		\left\{
    		\begin{aligned}
    			\frac{\partial}{\partial t} u(t,x)  -\divv(A(u(t,x)) \nabla u(t,x)) & = f(t,x) \\
    			u(0) & = u_0.
    		\end{aligned}
    		\right.
    	\end{equation*}
		Suppose that the following assumptions are satisfied.
		\begin{thm_enum}
			\item The coefficients $a_{ij}$ are $\beta$-Hölder continuous for some $\beta > \frac{1}{2}$.
			\item For all $M > 0$ there exist $\lambda(M) > 0$ such that for all $\abs{u} \le M$
				\begin{equation*}
					\Re \sum_{i,j=1}^n a_{ij}(u) \xi_i \overline{\xi_j} \ge \lambda(M) \abs{\xi}^2 \qquad \text{for all } \xi \in \IR^n.
				\end{equation*} 
		\end{thm_enum}
		If $q > n$ and $\beta > \frac{q}{2q - n}$, then there exists $C \ge 0$ such that for 
		\begin{equation*}
			\norm{f}_{L^p([0,T];L^q(\Omega))} + \norm{u_0}_{(D(-\divv A(u_0) \nabla \cdot),L^q(\Omega))_{\frac{1}{p},p}} \le C
		\end{equation*}
		the quasilinear problem \eqref{eq:QLP} has a solution 
		\begin{equation*}
			u \in W^{1,p}([0,T];L^q(\Omega)) \cap \BUC([0,T] \times \overline{\Omega})
		\end{equation*} 
		with $u(t) \in D(-\divv A(u(t,\cdot)) \nabla \cdot)$ for almost all $t \in [0,T]$ and $-\divv A(u) \nabla u \in L^p([0,T];L^q(\Omega))$. A fortiori, $u \in C^{\alpha - \frac{1}{p}}([0,T];C^{2(1-\alpha) - \frac{n}{q}}(\Omega))$ for $\alpha \in (\frac{1}{p},1-\frac{n}{2q})$.
	\end{theorem}
	\begin{proof}
		Choose $\alpha \in (\frac{1}{2\beta}, 1 - \frac{n}{2q})$, which is possible by our assumptions. Now, choose $\delta > 0$ with $\alpha - \delta > \frac{1}{2\beta}$ and $\alpha + \delta < 1 - \frac{n}{2q}$. Further, let 
		\begin{equation*}
			\mathcal{M} = \{ v \in W^{\alpha-\delta,p}([0,T];W_0^{2(1-\alpha-\delta),q}(\Omega)): v(0) = u_0 \}
		\end{equation*}
		and $\mathcal{M}_R$ for $R > 0$ be the ball $B(0,R)$ in $\mathcal{M}_R$. For $v \in \mathcal{M_R}$ consider the problem
    	\begin{equation*}
    		\label{eq:lp}
    		\tag{LP}
    		\LeftEqNo
    		\left\{
    		\begin{aligned}
    			\frac{\partial}{\partial t} u(t,x) - \divv(A(v(t,x)) \nabla u(t,x)) & = f(t,x) \\
    			u(0) & = u_0.
    		\end{aligned}
    		\right.
    	\end{equation*}
		Since $\alpha + \delta < 1 - \frac{n}{2q}$ and $\alpha - \delta > \frac{1}{2\beta} \ge \frac{1}{2}$, we have $v \in W^{\alpha-\delta,p}([0,T];\BUC(\Omega))$ and $\mathcal{M}$ is compactly embedded in $\BUC([0,T] \times \overline{\Omega})$. By the Arzelà--Ascoli theorem the functions in $\mathcal{M}_R$ are uniformly equicontinuous on $[0,T] \times \Omega$. As a consequence assumption~\ref{thm:mr_elliptic:vmo} of Theorem~\ref{thm:mr_elliptic} is satisfied and one can find uniform ellipticity constants for $A \circ v$ and $v \in \mathcal{M}_R$. For $\epsilon > 0$ with $r = (\alpha - \delta - \epsilon) \beta > \frac{1}{2}$ we have
    	\begin{align*}
    		\MoveEqLeft \norm{a_{ij} \circ v}^p_{\dot{W}^{r,p}([0,T]; L^{\infty}(\Omega))} = \int_0^T \int_0^T \frac{\norm{a_{ij}(v(t,\cdot)) - a_{ij}(v(s,\cdot))}_{\infty}^p}{\abs{t-s}^{1+pr}} \d s \d t \\
    		& \lesssim \int_0^T \int_0^T \frac{\norm{v(t,\cdot) - v(s,\cdot)}_{\infty}^{\beta p}}{\abs{t-s}^{1+pr}} \d s \d t = \norm{v}^{\beta p}_{\dot{W}^{r \beta^{-1}, \beta p}([0,T];L^{\infty}(\Omega))} \\
    		& = \norm{v}^{\beta p}_{\dot{W}^{\alpha - \delta - \epsilon, \beta p}([0,T];L^{\infty}(\Omega))} \lesssim \norm{v}^{\beta p}_{\dot{W}^{\alpha - \delta, p}([0,T];L^{\infty}(\Omega))}.
    	\end{align*}
		This means that the coefficients $A \circ v$ satisfy the assumptions of Theorem~\ref{thm:mr_elliptic}. Hence, \eqref{eq:lp} has a unique solution $u$ and there is $C(R) \ge 0$ independent of $v \in \mathcal{M}_R$ with
		\begin{align*}
			\norm{u}_{W^{1,p}([0,T];L^q(\Omega))} + \norm{\divv A(v) \nabla u}_{L^p([0,T];L^q(\Omega))} \le C (\norm{f}_{L^p([0,T];L^q(\Omega))} + \norm{u_0}).	
		\end{align*}
		Further, by the real interpolation formula for vector-valued Besov spaces~\cite[Corollary~4.3]{Ama00} one has for $\theta \in (\frac{1}{2}, 1-\frac{1}{2q})$ and sufficiently small $\epsilon > 0$ the for $v \in \mathcal{M}_R$ uniform embeddings
		\begin{equation}
			\label{eq:qlp:embedding}
			\begin{split}
				\MoveEqLeft W^{1,p}([0,T];L^q(\Omega)) \cap L^p([0,T];D(-\divv A \circ v \nabla \cdot)) \\
				& \hookrightarrow (L^p([0,T];D(-\divv A \circ v \nabla \cdot)), W^{1,p}([0,T];L^q(\Omega)))_{\theta,p} \\
				& = W^{\theta,p}([0,T];(D(-\divv A \circ v \nabla \cdot), L^q(\Omega))_{\theta,p}) \\
				& \hookrightarrow W^{\theta,p}([0,T];W_0^{2(1-\theta)-\epsilon,q}(\Omega)).
			\end{split}
		\end{equation}
		For the last step we may assume that the operators have uniformly bounded imaginary powers. In fact, this is true for a suitable shift, which does not influence the validity of the estimate, by the proof of Theorem~\ref{thm:mr_elliptic}. It then follows from the reiteration theorem for the real interpolation method~\cite[1.10.3, Theorem~2]{Tri78}, the identification of fractional domain spaces with complex interpolation spaces assuming bounded imaginary powers and the Kato square root estimate of Theorem~\ref{thm:kato} that
		\begin{align*}
			\MoveEqLeft (D(-\divv A \circ v \nabla \cdot), L^q(\Omega))_{\theta,p} = ([D(-\divv A \circ v \nabla \cdot), L^q(\Omega)]_{\frac{1}{2}}, L^q(\Omega))_{2\theta-1,p} \\
			& = (D((-\divv A \circ v \nabla \cdot)^{\frac{1}{2}}),L^q(\Omega))_{2\theta-1,p} = (W_0^{1,q}(\Omega),L^q(\Omega))_{2\theta-1,p} \\
			& = B_{0,p}^{2(1-\theta),q}(\Omega) \hookrightarrow W_0^{2(1-\theta)-\epsilon,q}(\Omega).
		\end{align*}
		All estimates hold uniformly for $v \in \mathcal{M}_R$, for the reiteration theorem one checks this with the help of~\cite[1.10.3, Theorem 1]{Tri78}. The embedding~\eqref{eq:qlp:embedding} implies that for small $\norm{f} + \norm{u_0}$ we obtain a well-defind map
    	\begin{align*}
    		S_R\colon \mathcal{M}_R & \to \mathcal{M}_R \\
    		v & \mapsto u_{}, \text{ where } u \text{ is the solution of~\eqref{eq:lp}}.
    	\end{align*}
    	It follows from~\eqref{eq:qlp:embedding} and the compact embedding results for vector-valued Sobolev spaces~\cite[Theorem~5.1]{Ama00} that $S_R \mathcal{M}_R$ is a precompact subset of $\mathcal{M}_R$. We next show that $S_R$ is continuous. For this let $v_n \to v$ in $\mathcal{M}_R$ and let $u_n = S_R v_n$. After passing to a subsequence we may assume that $v_n \to v$ in $\BUC([0,T] \times \Omega)$ and that $u_n$ converges weakly to some $u$ in
	\begin{equation*}
		W^{1,p}([0,T];L^q(\Omega)) \cap L^p([0,T];W_0^{1,q}(\Omega)).
	\end{equation*}
	Now, let $g \in L^{p'}([0,T];W_0^{1,q'}(\Omega))$. Note that $A^T(v_n) \nabla g \to A^T(v) \nabla g$ in $L^{q'}(\Omega)$ by the dominated convergence theorem. Since $u_n$ solves~\eqref{eq:lp} we have
    	\begin{align*}
    	 	0 & = \int_0^T \langle \dot{u}_n(t), g(t) \rangle \d t + \int_0^T \langle A(v_n(t)) \nabla u_n(t), \nabla g(t) \rangle \d t \\
    		& = \int_0^T \langle \dot{u}_n(t), g(t) \rangle \d t + \int_0^T \langle \nabla u_n(t), A^T(v_n(t)) \nabla g(t) \rangle \d t.
    	\end{align*}
    		Taking limits on both sides of the equation, we get
    	\begin{align*}
    		0 & = \int_0^T \langle \dot{u}(t), g(t) \rangle \d t + \int_0^T \langle A(v(t)) \nabla u(t), \nabla g(t) \rangle \d t
		\end{align*}
    	Since $g$ is arbitrary and $u_0 = u_n(0) \to u(0)$, this implies that $u$ solves~\eqref{eq:lp} on $W^{-1,q}(\Omega)$, i.e.\ is the unique integrated solution of~\eqref{eq:lp} given by Proposition~\ref{prop:wnacp}. Hence, $S_R v = u$. Since the same argument works for arbitrary subsequences, we have shown that $S$ is continuous. Now, by Schauder's fixed point there is some $u \in \mathcal{M}_R$ with $S_R u = u$. Using Theorem~\ref{thm:mr_elliptic} for $v = u$ we see that
		\begin{align*}
			\MoveEqLeft \norm{u}_{W^{1,p}([0,T];L^q(\Omega))} + \norm{\divv A(u) \nabla u}_{L^p([0,T];L^q(\Omega))} \\
			&  \le C (\norm{f}_{L^p([0,T];L^q(\Omega))} + \norm{u_0}_{(D(-\divv A(u_0) \nabla \cdot), L^q(\Omega))_{\frac{1}{p},p}}). \qedhere
		\end{align*}
	\end{proof}
	
	\begin{remark}
		We can only deduce the existence of solutions for small data because the constant in the maximal regularity estimate in Theorem~\ref{thm:mr_elliptic} depends on the $\VMO$-modulus of the coefficients and their fractional Sobolev norm. If one finds estimates on solutions of~\eqref{eq:QLP} independent of these regularity data, the Leray--Schauder principle would yield solutions for arbitrary $f$ and $u_0$.
		
		Further note that the above argument works for a far more general class of problems. For example, the coefficients $A(u)$ may depend in a non-local way on $u$, e.g.\ on the history of the solution, as in~\cite{Ama05} and~\cite{Ama06}.
	\end{remark}
		
\section{Optimality of the Results}\label{sec:optimality}

	In this section we will see that the maximal regularity results obtained in Theorem~\ref{thm:mr} are optimal or close to optimal. In fact, even in the form setting considered in Corollary~\ref{cor:form} maximal regularity may fail if one losens the assumed regularity, i.e.\ $\mathcal{A} \in \dot{W}^{\alpha,p}([0,T];\mathcal{B}(V,V'))$ for some $\alpha > \frac{1}{2}$ for maximal $L^p$-regularity. It was shown in~\cite[Theorem~5.1]{Fac16c} that there is a symmetric non-autonomous form with $\mathcal{A} \in C^{1/2}([0,T];\mathcal{B}(V,V'))$ and $f \in L^{\infty}([0,T];V)$ for which the unique solution given by Theorem~\ref{prop:wnacp} satisfies $\dot{u}(t) \not\in H$ for almost all $t \in [0,T]$ although $u \in L^{\infty}([0,T];V)$ holds. As a consequence maximal $L^p$-regularity fails for all $p \in [1, \infty]$. Note that $C^{1/2}([0,T];\mathcal{B}(V,V')) \hookrightarrow \dot{W}^{\alpha,q}([0,T];\mathcal{B}(V,V'))$ for all $\alpha \in (0,\frac{1}{2})$ and all $q \in [1, \infty]$. Hence, Theorem~\ref{thm:mr} fails for $\alpha < \frac{1}{2}$ in all possible variants.
	
	This leaves open the case $\alpha = \frac{1}{2}$ which is critical by the above results. Note that for $q \in (1,2)$ the space $\dot{W}^{1/2,q}([0,T];\mathcal{B}(V,V'))$ contains piecewise constant forms. Hence, as observed by Dier~\cite[Section~5.2]{Die14}, the failure of the Kato square root property for general forms implies that maximal $L^2$-regularity may not hold for $q < 2$. The example~\cite[Example~7.2]{Fac16d} shows that for $p > 2$ maximal $L^p$-regularity in the range $q \in (1, 2)$ does not even hold for elliptic operators. Note that for $p \in (1,2)$ these arguments based on the imcompability of trace spaces break down.
	
	Refining the arguments in~\cite{Fac16c}, we show that for symmetric forms maximal $L^p$-regularity may fail for $p \in [1, \infty]$ and $\mathcal{A} \in \dot{W}^{1/2,q}([0,T];\mathcal{B}(V,V'))$ for some $q > 2$.
	
	\begin{example}
		We take $H = L^2([0,1/2])$ and $V = L^2([0,1], w)$ with $w(x) = (x \abs{\log x})^{-3/2}$. Further, we consider $u(t,x) = c(x) (\sin(t \phi(x)) + d)$ for $\phi(x) = w(x)$, $c(x) = x \cdot \abs{\log x}$ and some sufficiently large $d > 0$. Note that for all $t \in [0,T]$
		\begin{equation*}
			\norm{\dot{u}(t)}_H^2 \simeq \int_0^{1/2} \abs{c(x) \phi(x)}^2 \d x = \int_0^{1/2} x^{-1} \frac{1}{\abs{\log x}} \d x = \infty.
		\end{equation*}
		Hence, $\dot{u}(t) \not\in H$ for all $t \in [0,T]$. Following the ideas and arguments in \cite{Fac16c} we now show that $u$ is indeed an integrated solution of a non-autonomous problem associated to some bounded coercive symmetric sesquilinear form $a\colon [0,T] \times V \times V \to \IC$ and the inhomogeneous part $f(t) = u(t) \in L^{\infty}([0,T];V)$. For this one can use the same extension procedure as in~\cite[Section~4]{Fac16c}. Following~\cite[Section~5]{Fac16c} it then remains to check the regularity of the extended forms. Since $\dot{W}^{\alpha,p} \cap L^{\infty}$ is an algebra under pointwise multiplication, the regularity question boils down to the regularity of the mapping $u\colon [0,T] \mapsto V$. 
		
		We now show explicitly that $u \in \dot{W}^{1/2,q}([0,T];V)$ for all $q \in (2, \infty)$. Note that one the one hand
		\begin{equation}
			\label{eq:sin_estimate}
			\abs{\sin(t\phi(x)) - \sin(s\phi(x))}^2 \le \abs{t-s}^2 \phi^2(x) = \abs{t-s}^2 x^{-3} \abs{\log x}^{-3}.
		\end{equation}
		On the other hand the left hand side can clearly be estimated by $4$. Now, let $\psi(x) = 2x^{3/2} \abs{\log x}^{3/2}$. Then~\eqref{eq:sin_estimate} gives the sharper estimate if and only if $\abs{t-s} \le \psi(x)$ or equivalently $x \ge \psi^{-1}(r)$. Splitting the fractional norm, we obtain
		\begin{equation}
			\label{eq:frac_estimate}
			\begin{split}
    			& \biggl( \int_0^{T} \int_0^{T} \frac{\norm{u(t) - u(s)}^q_V}{\normalabs{t-s}^{1+\frac{q}{2}}} \d t \d s \biggr)^{1/q} \\
				& = \biggl( \int_0^{T} \int_{-t}^{T-t} \frac{\norm{u(t) - u(t+r)}_V^q}{\normalabs{r}^{1+\frac{q}{2}}} \d r \d t \biggr)^{1/q} \\
    			& \lesssim \biggl( \int_0^T \int_{-t}^{T-t} \biggl( \int_0^{\psi^{-1}(\abs{r})} x^{1/2} \abs{\log x}^{1/2} \d x \biggr)^{q/2} \frac{\d r}{\normalabs{r}^{1+\frac{q}{2}}} \d t \biggr)^{1/q} \\
    			& \qquad + \biggl( \int_0^T \int_{-t}^{T-t} \biggl( \int_{\psi^{-1}(\abs{r})}^{1/2} x^{-5/2} \abs{\log x}^{-5/2} \d x \biggr)^{q/2} \frac{\d r}{\normalabs{r}^{1-\frac{q}{2}}} \d t \biggr)^{1/q}.
			\end{split}
		\end{equation}
		Now, for the innermost integral of the first term we have for $F(x) = x^{3/2} \abs{\log x}^{1/2}$
		\begin{align*}
			\MoveEqLeft \int_0^{\psi^{-1}(\abs{r})} x^{1/2} \abs{\log x}^{1/2} \d x \lesssim \int_0^{\psi^{-1}(\abs{r})} F'(x) \d x = F(\psi^{-1}(\abs{r})) \\
			& \lesssim \psi(\psi^{-1}(\abs{r})) \normalabs{\log \psi^{-1}(\abs{r})}^{-1} = \abs{r} \normalabs{\log \psi^{-1}(\abs{r})}^{-1} \lesssim \abs{r} \abs{\log r}^{-1}.
		\end{align*}
		Analogously, for the second term we have  for $F(x) = -x^{-3/2} \abs{\log x}^{-5/2}$
		\begin{align*}
			\MoveEqLeft \int_{\psi^{-1}(\abs{r})}^{1/2} x^{-5/2} \abs{\log x}^{-5/2} \d x \lesssim \int_{\psi^{-1}(\abs{r})}^{1/2} F'(x) \d x \le - F(\psi^{-1}(\abs{r})) \\
			& \lesssim \frac{1}{\psi(\psi^{-1}(\abs{r}))} \normalabs{\log \psi^{-1}(\abs{r})}^{-1} \lesssim \abs{r}^{-1} \abs{\log r}^{-1}.
		\end{align*}
		Hence, \eqref{eq:frac_estimate} is dominated up to a constant by the for $q > 2$ finite expression
		\begin{align*}
			\biggl( \int_{0}^{T} \abs{\log r}^{-q/2} \frac{\d r}{\normalabs{r}} \biggr)^{1/q}.
		\end{align*}
	\end{example}
	
	Hence, for maximal $L^2$-regularity of forms the only case left open is that of $\dot{W}^{1/2,2}([0,T];\mathcal{B}(V,V'))$ regularity which we are not able to answer at the moment. Note that there is also a positive result assuming some half differentiability. Namely, it was shown by Auscher and Egert~\cite{AusEge16} that for elliptic operators one has maximal $L^2$-regularity if the coefficients $a_{ij}$ satisfy $\partial^{1/2} a_{ij} \in \BMO$. This in particular implies $a_{ij} \in \dot{H}^{1/2,q}$ for all $q \in (1, \infty)$, which in turn implies $a_{ij} \in \dot{W}^{1/2,q}$ for all $q \ge 2$, which in general is not sufficient for maximal $L^p$-regularity by the above example. In the other direction the inclusion $\dot{W}^{1/2,q} \hookrightarrow \dot{H}^{1/2,q}$ does only hold for $q \in (1, 2]$. Hence, for $q \in (1,2)$ the space $\dot{H}^{1/2,q}$ contains step functions. Note that in the critical case one has $\dot{H}^{1/2,2} = \dot{W}^{1/2,2}$, i.e.\ the Besov and the Bessel scale give rise to the same problem.
						
\emergencystretch=0.75em
\printbibliography

\end{document}